\documentclass[12pt]{amsart}

\usepackage[latin1]{inputenc}   
\usepackage{amssymb,amsmath,amsthm}
\usepackage{latexsym}

\usepackage[all]{xy}
\usepackage{graphicx}

\usepackage{xcolor}

\usepackage{verbatim}

\usepackage{tikz}
\usetikzlibrary{arrows,decorations.pathmorphing,decorations.markings,backgrounds,positioning,fit}

\input xy
\xyoption{all}

\newcommand{\Q}{\ensuremath{\mathbb{Q}}}

\newcommand{\x}{\ensuremath{\mathbf{x}}}

\DeclareMathOperator{\Spec}{Spec}

\DeclareMathOperator{\Hom}{Hom}

\newtheorem{Thm}{Theorem}[section]

\newtheorem{Cor}[Thm]{Corollary}

\newtheorem{Lem}[Thm]{Lemma}

\newtheorem{Prop}[Thm]{Proposition}

{\theoremstyle{definition}
       \newtheorem{defi}[Thm]{Definition}

       \newtheorem{Rmk}[Thm]{Remark}
       \newtheorem{ex}[Thm]{Example}

\numberwithin{equation}{section}

\tikzstyle{mutable}=[inner sep=0.5mm,circle,draw,minimum size=2mm]
\tikzstyle{frozen}=[inner sep=.9mm,rectangle,draw]
\tikzstyle{dot} = [fill=black!25,inner sep=0.5mm,circle,draw,minimum size=1mm]
\tikzstyle{marked}=[inner sep=0.5mm,circle,draw,blue!75!black,fill=blue!50]

\title{Singularities of Locally Acyclic Cluster Algebras}
\author{Ang\'elica Benito, Greg Muller, Jenna Rajchgot, and Karen E. Smith}
 \email{abenitos@umich.edu, morilac@umich.edu, rajchgot@umich.edu,  kesmith@umich.edu}

\thanks{The last author  acknowledges the financial support of NSF grant DMS-1001764 and the Clay Foundation. The first author is partially support by MTM2012-35849}
\address{Department of Mathematics
 University of Michigan\\
 Ann Arbor, MI 48109}

\begin{document}
\maketitle

\begin{abstract}
We show that locally acyclic cluster algebras have (at worst) canonical singularities. In fact, we prove that locally acyclic cluster algebras of positive characteristic are strongly $F$-regular. In addition, we show that upper cluster algebras are always Frobenius split by a canonically defined splitting, and that they have a free canonical module of rank one. We also give examples to show that not all upper cluster algebras are F-regular if the local acyclicity is dropped.
\end{abstract}

\section{Introduction}
Fomin and Zelevinsky introduced cluster algebras at the close of the twentieth century as a way to study total positivity in a wide range of contexts. Since then, cluster algebra structures have been discovered in many unexpected corners of mathematics (and physics), including Teichm\"uller theory \cite{GSV05,FG07}, discrete integrable systems \cite{FZ03y}, knot theory \cite{MulSk,MSW12}, and mirror symmetry \cite{GS14,KS13}, just to name a few.

Locally acyclic cluster algebras, introduced recently by the second author \cite{MulLA},  are a large class of cluster algebras which are simultaneously flexible enough to include many interesting examples --- including many fundamental examples from representation theory and most examples from Teichm\"uller theory --- yet restrictive enough to avoid the pathological behavior sometimes found in general cluster algebras. For example, locally acyclic cluster algebras are finitely generated and normal, while a general cluster algebra may fail to be either.
The main theorem of this paper is that {\it locally acyclic cluster algebras have (at worst) canonical singularities.} In fact, we show that locally acyclic cluster algebras of prime characteristic are {\it strongly $F$-regular,}  a strong form of  Frobenius split  which  implies many nice restrictions on the singularities; for example, $F$-regular varieties are normal, Cohen-Macaulay, pseudorational, and  have Kawamata log terminal singularities (if the canonical class divisor is $\Bbb Q$-Cartier) or canonical singularities (if the canonical class is Cartier). These characteristic $p$ results imply the corresponding statements in characteristic zero as well.  For  survey, see  e.g. \cite{Smi97b} or \cite{SmithMSRIsurvey}. 
 
 Associated to a cluster algebra $\mathcal{A}$ is its \emph{upper cluster algebra}  $\mathcal{U}$. These related algebras have the same fraction field and satisfy $\mathcal{A} \subseteq \mathcal{U}$  (cf. \cite{BFZ05}). 
 We show that all upper cluster algebras in positive characteristic have a `cluster' Frobenius splitting, which can be expressed explicitly in terms of any cluster. We also prove the closely related result that upper cluster algebras have a free canonical module, which is generated by any `log volume form' in a cluster of cluster variables. The latter of these results is found in the appendix.

The inclusion $\mathcal{A} \subseteq \mathcal{U}$ need not be equality, though it is in the case when $\mathcal{A}$ is locally acyclic \cite{MulAU}. When equality fails, a general philosophy is that $\mathcal{U}$ is better behaved than $\mathcal{A}$. In this direction, we show that if an upper cluster algebra $\mathcal{U}$ fails to be $F$-regular, then $\mathcal{A}$ also fails to be $F$-regular, and we provide an example of this situation. Taking the ground field to be of characteristic zero, this gives an example of a finitely generated upper cluster algebra $\mathcal U$ which has {\it log canonical  but not log terminal singularities.} 
 We also provide an example where $\mathcal{A}\neq \mathcal{U}$ and $\mathcal{A}$ is pathological (e.g. $\mathcal{A}$ is non-Noetherian), but $\mathcal{U}$ is nevertheless strongly $F$-regular.

All of our results and arguments are also valid for cluster algebras given by an arbitrary skew-symmetrizable matrix. However, we have written the exposition in the slightly less general setting of cluster algebras given by quivers for the sake of accessibility. Experts will have no trouble adapting the arguments to the more general setting. 


\subsection*{Acknowledgements} We are grateful to MRSI for financial support and for organizing the workshop in August 2012 designed to foster interactions between commutative algebra and cluster algebras, where this collaboration was born.

We are also indebted to Allen Knutson and David Speyer, who graciously shared their unpublished work on Frobenius splittings of cluster algebras.

\section{Cluster algebras}

\def\ZP{\mathbb{ZP}}
\def\x{\mathbf{x}}
\def\y{\mathbf{y}}
\def\B{\mathsf{B}}
\def\ZZ{\mathbb{Z}}
\def\A{\mathcal{A}}
\def\U{\mathcal{U}}
\def\k{\mathsf{k}}
\def\Q{\mathsf{Q}}
\def\k{\mathsf{k}}


A cluster algebra is a commutative domain with some extra combinatorial structure. It  comes equipped with a (usually infinite) set of generators called \emph{cluster variables}, which can be recursively generated from a \emph{seed}: a quiver decorated with a free generating set for a field.

We will consider cluster algebras over an arbitrary field $\k$, although in the literature they are usually defined over $\Bbb Q$, $\Bbb R$ or $\Bbb Z$.  The choice of scalars is mostly irrelevant to the definitions, and most proofs of standard results go through without change.  As such, we will cite the original results without comment, and only address the differences  as needed.

\subsection{Seeds and mutations}

Let $\k$ be a field, and let $\mathcal{F}$ be a purely transcendental finite extension of $\k$.
A \textbf{seed}  for $\mathcal{F}$ over $\k$ consists of the following data.
\begin{itemize}
	\item A \textbf{quiver} $\Q$ without loops or directed 2-cycles.
	\item A  
	 bijection from the vertices of $\Q$ to a  set of algebraically independent generators $\x=\{x_1,x_2,...,x_n\}$ for $\mathcal{F}$ over $\k$.  The image $x_i$ of a vertex $i$ is called the \textbf{cluster variable} at that vertex, and the set $\x$ is called a \textbf{cluster}.
	\item A subset of the vertices of $\Q$ designated as \textbf{frozen}; the rest are called \textbf{mutable}.  We impose the non-standard convention that every vertex which touches no arrow is frozen.\footnote{This convention allows us to define cluster algebras in characteristic two, and otherwise produces the same definition as the ``usual" convention in every other characteristic. The point is that this convention  prevents the numerator in the mutation formula \eqref{eqmut} from being 2, which in characteristic two would mean that a mutation at that vertex  would {\it not produce} another valid cluster variable.}
\end{itemize}

Seeds will usually be denoted as a pair $(\Q,\x)$, with the other data suppressed.  The number of vertices of $\Q$ (denoted $n$ hereafter) is the \textbf{rank} of the seed, and the number of mutable vertices (denoted $m$ hereafter) is the \textbf{mutable rank}.

Seeds may be drawn as a quiver with circles \begin{tikzpicture}
\node[mutable] at (0,0) {};
\end{tikzpicture} for mutable vertices, and rectangles \begin{tikzpicture}
\node[frozen] at (0,0) {};
\end{tikzpicture} for frozen vertices,  each with the corresponding cluster variable inscribed (e.g. Figure \ref{fig: exampleseed}).

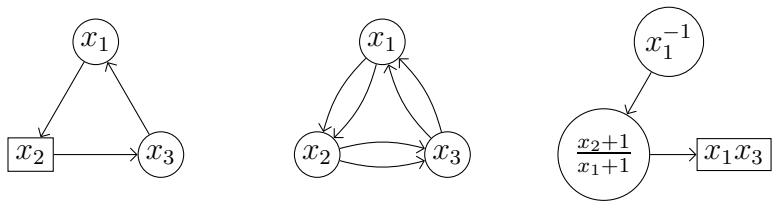
\begin{figure}[h!t]
	\begin{tikzpicture}
	\begin{scope}[xshift=0in]
		\node[mutable] (x1) at (90:1) {$x_1$};
		\node[frozen,inner sep=1mm] (x2) at (210:1) {$x_2$};
		\node[mutable] (x3) at (-30:1) {$x_3$};
		\draw[-angle 90] (x1) to (x2);
		\draw[-angle 90] (x2) to (x3);
		\draw[-angle 90] (x3) to (x1);
	\end{scope}
	\begin{scope}[xshift=1.5in]
		\node[mutable] (x1) at (90:1) {$x_1$};
		\node[mutable] (x2) at (210:1) {$x_2$};
		\node[mutable] (x3) at (-30:1) {$x_3$};
		\draw[-angle 90,relative, out=15,in=165] (x1) to (x2);
		\draw[-angle 90,relative, out=-15,in=-165] (x1) to (x2);
		\draw[-angle 90,relative, out=15,in=165] (x2) to (x3);
		\draw[-angle 90,relative, out=-15,in=-165] (x2) to (x3);
		\draw[-angle 90,relative, out=15,in=165] (x3) to (x1);
		\draw[-angle 90,relative, out=-15,in=-165] (x3) to (x1);
	\end{scope}
	\begin{scope}[xshift=3in]
		\node[mutable] (x1) at (90:1) {$x_1^{-1}$};
		\node[mutable,inner sep=1mm] (x2) at (210:1) {$\frac{x_2+1}{x_1+1}$};
		\node[frozen] (x3) at (-30:1) {$x_1x_3$};
		\draw[-angle 90] (x1) to (x2);
		\draw[-angle 90] (x2) to (x3);
	\end{scope}
	\end{tikzpicture}
	\caption{Three examples of seeds in $\mathcal{F}=\k(x_1,x_2,x_3)$.}
	\label{fig: exampleseed}
\end{figure}

A seed $(\Q,\x)$ may be \textbf{mutated} at any mutable vertex $k$ to produce a new seed $(\mu_k(\Q),\mu_k(\x))$ for $\mathcal F$.  The quiver $\mu_k(\Q)$ is constructed in three steps.

\noindent \begin{minipage}[l]{3.2in}
	\vspace{.1in}
	\begin{enumerate}
		\item For each pair of arrows $i\rightarrow k \rightarrow j$ through the vertex being mutated, add an arrow $i\rightarrow j$.
		\item Reverse the orientation of every arrow incident to $k$.
		\item Cancel any directed 2-cycles in pairs.
	\end{enumerate}
\end{minipage}
\hspace{.3in}
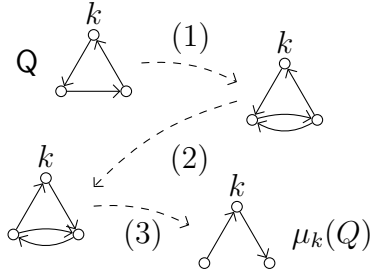
\begin{minipage}[l]{1.5in}
	\begin{tikzpicture}
	\node at (-.1in,.05in) {$\Q$};
	\begin{scope}[xshift=.25in,yshift=0in,scale=.5]
		\node[dot, fill=white] (x1) at (90:1) {};
		\node[dot, fill=white] (x2) at (210:1) {};
		\node[dot, fill=white] (x3) at (-30:1) {};
		\node[above] at (x1) {$k$};
		\draw[-angle 90] (x1) to (x2);
		\draw[-angle 90] (x2) to (x3);
		\draw[-angle 90] (x3) to (x1);
	\end{scope}
	\draw[dashed, -angle 90,relative, out=15,in=165] (.5in,.05in) to node[above] {(1)} (1.0in,-.05in);
	\begin{scope}[xshift=1.25in,yshift=-.15in,scale=.5]
		\node[dot, fill=white] (x1) at (90:1) {};
		\node[dot, fill=white] (x2) at (210:1) {};
		\node[dot, fill=white] (x3) at (-30:1) {};
		\node[above] at (x1) {$k$};
		\draw[-angle 90] (x1) to (x2);
		\draw[-angle 90] (x3) to (x1);
		\draw[-angle 90,relative, out=15,in=165] (x2) to (x3);
		\draw[angle 90-,relative, out=-30,in=-150] (x2) to (x3);
	\end{scope}
	\draw[dashed, -angle 90,relative, out=-15,in=195] (1.0in,-.15in) to node[below right] {(2)} (.25in,-.6in);
	\begin{scope}[xshift=0in,yshift=-.75in,scale=.5]
		\node[dot, fill=white] (x1) at (90:1) {};
		\node[dot, fill=white] (x2) at (210:1) {};
		\node[dot, fill=white] (x3) at (-30:1) {};
		\node[above] at (x1) {$k$};
		\draw[-angle 90] (x2) to (x1);
		\draw[-angle 90] (x1) to (x3);
		\draw[-angle 90,relative, out=15,in=165] (x2) to (x3);
		\draw[angle 90-,relative, out=-30,in=-150] (x2) to (x3);
	\end{scope}
	\draw[dashed, -angle 90,relative, out=15,in=165] (.25in,-.7in) to node[below] {(3)} (.75in,-.8in);
	\begin{scope}[xshift=1.0in,yshift=-.9in,scale=.5]
		\node[dot, fill=white] (x1) at (90:1) {};
		\node[dot, fill=white] (x2) at (210:1) {};
		\node[dot, fill=white] (x3) at (-30:1) {};
		\node[above] at (x1) {$k$};
		\draw[-angle 90] (x2) to (x1);
		\draw[-angle 90] (x1) to (x3);
	\end{scope}
	\node at (1.5in,-.85in) {$\mu_k(Q)$};
	\end{tikzpicture}
\end{minipage}

The cluster variables in $\mu_k(\mathbf{x})$ are the same as those in $\mathbf{x}$, except for the cluster variable at vertex $k$, which becomes
\begin{equation}\label{eqmut}
 x_k' := \left(\displaystyle \prod_{i\rightarrow k} x_i +\prod_{k\rightarrow i} x_i\right) x_k^{-1},
\end{equation}
where the products are over all arrows into or out of $k$, respectively.  Frozen vertices stay frozen.

Mutating at the same vertex twice in a row returns  the original seed. That is, mutation is an involution on the set of seeds of $\mathcal F/\k$. Two seeds are \textbf{mutation-equivalent} if they are related by a sequence of mutations.

\subsection{Cluster algebras}  Fix a seed $(\Q,\x)$ for $\mathcal F$ over $\k$. The union of all the clusters which appear in mutation-equivalent seeds defines the complete set of \textbf{cluster variables} in the ambient field $\mathcal{F}$, naturally grouped into overlapping {\bf clusters} consisting of those appearing together in a seed.  The \textbf{cluster algebra} $\A(\Q,\x)$ determined by $(\Q,\x)$ is the sub-$\k$-algebra of $\mathcal{F}$ generated by all of the cluster variables and the inverses of the frozen variables.    The cluster algebra only depends on the mutation-equivalence class of the initial seed, and so the initial seed $(\Q,\x)$ will often be omitted from the notation.

A fundamental property of cluster algebras is the \emph{Laurent phenomenon} \cite{FZ02}, which states that each cluster variable can be expressed as a Laurent polynomial in {\it any} cluster. Put differently, 
the localization of $\A$ at any cluster $\x=\{x_1,x_2,...,x_n\}$ is the ring of Laurent polynomials in $\x$ over $\k$:
\[ \A\hookrightarrow \A[x_1^{-1},x_2^{-1},...,x_n^{-1}] = \k[x_1^{\pm1},x_2^{\pm1},...,x_n^{\pm1}] \subset \mathcal{F}.\]
Every cluster in $\A$ defines such an inclusion.  This naturally leads to the following definition: the \textbf{upper cluster algebra} $\U$ of $\A$ is the intersection of each of these Laurent rings, taken inside the ambient field $\mathcal{F}$.
\[\U:= \bigcap_{\text{clusters } \x\subset \A} \k[x_1^{\pm1},x_2^{\pm1},...,x_n^{\pm1}]\subset \mathcal{F}.\]
By the Laurent phenomenon, there is an inclusion $\A\subseteq \U$.  This inclusion is not always equality (see \cite[Proposition 1.26]{BFZ05}), but it is an equality in all the simplest examples, and in many of the most important examples.

\begin{Lem}
Upper cluster algebras are normal.
\end{Lem}
\begin{proof}
Laurent rings over fields are regular, and hence normal.  Since an intersection of normal rings inside their common fraction field is normal, upper cluster algebras are normal.
\end{proof}

\subsection{Cluster localization.}

Under certain circumstances, localizing a cluster algebra at one or more cluster variables is again a cluster algebra.  This important idea is first discussed in  \cite{MulLA} and further developed in \cite{MulAU},  to which we refer for more details. 

Given a seed $(\Q,\x)$ over $\k$ and a designated subset $\{k_1,k_2,...,k_a\}$ of its mutable vertices, we can make a new seed $(\Q^\dagger,\x^\dagger)$ by making those vertices frozen.
Because mutations for $\Q^{\dagger}$ are all mutations for $\Q$, 
there is a natural containment
\begin{equation}\label{clusterLoc}
 \A(\Q^\dagger,\x^\dagger) \subseteq \A(\Q,\x)[x_{k_1}^{-1},x_{k_2}^{-1},\dots ,x_{k_a}^{-1}] \end{equation}
If this is an equality, $\A(\Q^\dagger,\x^\dagger)$ is called a \textbf{cluster localization} of $\A(\Q,\x)$.

Although it can be difficult to determine whether a particular 
localization is  a cluster localization,  there is one situation where it is easy.  Indeed, we have inclusions
\begin{equation}\label{clusterLocinclusion}
 \A(\Q^\dagger,\x^\dagger) \subseteq \A(\Q,\x)[x_{k_1}^{-1},x_{k_2}^{-1},\dots ,x_{k_a}^{-1}] 
\subseteq \U(\Q, \x)[x_{k_1}^{-1},x_{k_2}^{-1},\dots ,x_{k_a}^{-1}] \subseteq \U(\Q^{\dagger}, \x^{\dagger}),
 \end{equation}
where the first and third inclusion follow from the fact that the mutations for $\Q^{\dagger}$ are a subset of the mutations for $\Q$ (and the middle inclusion follows from the Laurent phenomenon  for $\A(\Q)).$
 Thus the inclusion in (\ref{clusterLoc}) is always equality whenever $\mathcal{A}(\Q^\dagger, \x^\dagger)=\U(\Q^\dagger, \x^\dagger)$. One extreme case is where we freeze all vertices: since obviously, $\A = \U$ when no mutations can happen, it follows that localizing at any full cluster $\x$
is a cluster localization. More generally,  $\mathcal{A}(\Q^\dagger,\x^\dagger)=\U(\Q^\dagger,\x^\dagger)$  will necessarily hold if `enough' mutable vertices become frozen. 

For example, if we freeze enough variables to break any directed cycles in $\Q$, we arrive at an
{\bf acyclic} quiver $\Q^{\dagger}$. By definition, a quiver is acyclic if it  has no directed cycles through mutable vertices;  a {\bf cluster algebra is acyclic} if it admits some acyclic seed.  Because acyclic cluster algebras are known to equal their upper cluster algebras (by Theorem \ref{AcyclicPres} below), the chain of inclusions (\ref{clusterLocinclusion})  above implies that $\A(\Q^{\dagger}, \x^{\dagger})$ is a cluster localization whenever     $\Q^{\dagger}$ is  acyclic.

\subsection{Cluster Covers}
The idea of cluster localization is powerful when a cluster algebra can be covered by cluster localizations. \begin{defi}
For a cluster algebra $\A$, a set $\{\A_i\}_{i\in I}$ of cluster localizations of $\A$ is called a \textbf{cluster cover} if the corresponding open subschemes cover $\Spec(\A)$,  that is, if
\[ \Spec(\A) = \bigcup_{i\in I} \Spec(\A_i). \]
\end{defi}
If a cluster algebra admits a cluster cover, 
any ``geometric" property, such as normality,  smoothness, or even different classes of singularities,
can be checked locally on the cluster localizations. 
Another  property which may be checked on a cover is whether $\A=\U$:
\begin{Lem}  \cite[Lemma 3.3.2]{MulAU}
If $\{\A_i\}_{i\in I}$ is a cluster cover of  $\A$, and $\A_i=\U_i$ for each $i\in I$, then $\A=\U$. 
\end{Lem}

A powerful observation proposed in  \cite{MulLA} is that many notable classes of cluster algebras admit covers by \emph{acyclic cluster algebras}.
\begin{defi}
A cluster algebra is \textbf{locally acyclic} if it admits a cluster cover by acyclic cluster algebras.
\end{defi}

The class of locally acyclic cluster algebras is {\it much wider} than the class of 
acyclic cluster algebras. The latter class is well-understood and very nicely behaved, but far too restrictive to be itself a major class. 
On the other hand, {\it locally} acyclic cluster algebras include, for example, cluster algebras of Grassmannians,  cluster algebras of \emph{marked surfaces} with at least two marked points on the boundary \cite[Theorem 10.6]{MulLA}, as well as  cluster algebras of double Bruhat cells and more generally, \emph{positroid cells} \cite{MS14}.
Because the geometric properties of  locally acyclic cluster algebras follow nicely from the acyclic case, there is now substantial interest in identifying locally acyclic cluster algebras.

\begin{Prop}[\cite{MulLA}]\label{propLA}
A locally acyclic cluster algebra over $\k$ is finitely generated over $\k$ and equal to its  upper cluster algebra.  A locally acyclic cluster algebra is normal and a local complete intersection (hence Gorenstein). 
\end{Prop}

This follows with little fuss from the acyclic case, due to Berenstein, Fomin, and Zelevinsky.\footnote{The proof in \cite{BFZ05} assumes an additional condition, that the cluster algebra is `totally coprime'.  However, it was shown in \cite{MulAU} that this condition is unnecessary.}
\begin{Thm}[Corollary 1.17 \cite{BFZ05}, Corollary 4.2.2 \cite{MulAU}]\label{AcyclicPres}
Let $(\Q,\x)$ be an acyclic seed.  Then the cluster algebra $\A(\Q)$ is a finitely generated complete intersection, equal to its upper cluster algebra  $ \U(\Q).$
\end{Thm}

\begin{Rmk}
It is important to note that not every cluster algebra admits a cover by proper cluster localizations. For example, the Markov cluster algebra generated from the middle seed in Figure 1 can not by covered by proper cluster localizations. Indeed, one easily checks that $\A$ can be $\Bbb N$-graded, with every cluster variable homogeneous of degree one. So, any non-trivial cluster localization $ \Spec(\A_i)$ of  $ \Spec(\A)$ necessary misses the unique homogeneous maximal ideal generated by the cluster variables.
\end{Rmk}

\section{Frobenius splittings}


\subsection{Frobenius splittings}

Every domain{\footnote{The assumption that $R$ is a domain is completely unnecessary, but it simplifies our discussion and is sufficient for our purposes.}} $R$ over a field of positive characteristic $p$ has a canonical ring map, the \textbf{Frobenius endomorphism}
$$F: R \rightarrow R, \,\,\, {\rm defined \,\,by} \,\,\, x \mapsto x^p.
$$
 The Frobenius map is  an $R$-module map if we equip the target copy of $R$ with the $R$-module 
structure it gets via  restriction of scalars. In practice, it is convenient to denote the target copy of $R$ by some other notation. We  denote  the target copy by $R^{1/p}$ and its elements by $r^{1/p}$, which is consistent with viewing the target copy of $R$ as (the canonically isomorphic ring) $R^{1/p}$ inside the algebraic closure of the fraction field of $R$. In this case, the elements of $r$ act on elements $x^{1/p} \in R^{1/p}$ by $r\cdot x^{1/p}=(r^px)^{1/p},$  the usual multiplication $rx^{1/p}$ in the fraction field.
In this notation, the Frobenius map becomes the inclusion  
\begin{equation}\label{eq: Frobeniusmodule}
\xymatrix@C=0.9cm@R=0cm{
 R\ \ar@{^(->}[r]^{F} & R^{1/p}\\
 r\ar@{|->}[r] &(r^p)^{1/p} = r\!\!\!\!\!\!\!\!\!\!\!\!\!\!\!
 }
\end{equation}
of $R$ into the overring $R^{1/p}$ of $p$-th roots.

We say that $R$ is {\bf $F$-finite}  if $R^{1/p}$ is a finitely generated $R$-module. This is a fairly weak condition, satisfied for example, by every finitely generated algebra over a perfect field $\k$.

A famous theorem of Kunz states that an $F$-finite domain $R$  is regular  if and only if $R^{1/p}$ is locally free over $R$ (\cite[Theorem 2.1]{Kun69}).  More generally, one should expect that the closer $R^{1/p}$ is to being locally free over $R$, the milder the singularities of $R$. Frobenius split rings and strongly $F$-regular rings are examples of rings in which some degree of ``freeness" is retained of $R^{1/p}$ over $R$.

\begin{defi}
A domain $R$ is \textbf{Frobenius split} if the map \eqref{eq: Frobeniusmodule} splits in the category of $R$-modules.  
A choice of splitting $\phi:R^{1/p}\rightarrow R$ is called a \emph{Frobenius splitting}.
\end{defi}

\begin{ex}
 Every field $\k$ of characteristic $p$ is  Frobenius split, since $\k^{1/p}$ is a vector space over the subfield $\k$.  
  For a perfect field $\k$, the Frobenius endomorphism is a field isomorphism, and its inverse is   the unique Frobenius splitting of $\k$.  
\end{ex}


\begin{ex} Polynomial rings are Frobenius split. Define the \textbf{standard splitting} of the polynomial ring $\k[x_1,x_2, ...,x_n]$ to be 
 given by
\[ \phi:(\k[x_1, x_2,...,x_n])^{1/p} \longrightarrow \k[x_1, x_2, ...,x_n], \]
\[ \phi\left(( \lambda x_1^{a_1}x_2^{a_2}\cdots x_n^{a_n})^{1/p}\right) = \left\{ \begin{array}{cc}
\phi(\lambda^{1/p})x_1^{a_1/p}x_2^{a_2/p}\cdots x_n^{a_n/p} & \text{if } a_1,a_2,...,a_n\in p\ZZ \\
0 & \text{otherwise} \\
\end{array} \right\},\]
where, the map 
$\phi: \k^{1/p} \rightarrow \k$ on scalars $\lambda$ is taken to be any fixed splitting of Frobenius.
\end{ex} 

\begin{Rmk}
The standard splitting of a polynomial ring is a Frobenius splitting, and will be the source of Frobenius splittings of cluster algebras. 
It depends on a choice of generators $\x$, and if $\k$ is not perfect, it depends on a choice of Frobenius splitting for $\k$. We suppress the dependence on the choice of a Frobenius splitting of $\k$ by assuming our ground field comes with a fixed Frobenius splitting. In any case, when $\k$ is perfect, there is a unique splitting.
\end{Rmk}

The standard splitting of a polynomial ring induces a splitting, also called the \textbf{standard splitting},  of the $\k$-Laurent ring $L = \k[x_1^{\pm1}, \dots, x_n^{\pm1}]$, using the exact same formula as above. 
An isomorphism between two $\k$-Laurent rings will commute with the standard splitting, so it does not depend on a choice of presentation. 
 
The standard splitting of a $\k$-Laurent ring has the following universal property.

\begin{Lem}\label{lemma: uniLaurent}
If $L$ is a finitely-generated Laurent ring over a perfect{\footnote{Perfect is not necessary here but it suffices for our purposes and simplifies the discussion. }}
 field $\k$ of characteristic $p$, then the standard splitting $\phi$ freely generates $\Hom_L(L^{1/p},L)$ as an $L^{1/p}$-module.
\end{Lem}

Explicitly, every $L$-module map $L^{1/p}\rightarrow L$ (including every Frobenius splitting) can be written as the composition
$$
L^{1/p} \overset{m_s}\longrightarrow  L^{1/p} \overset{\phi}\longrightarrow L
$$
 of the standard splitting $\phi$ and ``multiplication by $s^{1/p}$"  map $m_s$ for some {\it unique} $s \in L$.  We denote this composition by $\phi \circ s^{1/p}$.

\begin{proof} 
Let $L=\k[x_1^{\pm1},x_2^{\pm1},...,x_n^{\pm1}]$.
As an $L$-module, $L^{1/p}$ has a basis consisting of monomials $\mathbf{x}^{\mathbf{a}} =x_1^{a_1}x_2^{a_2}\cdots x_n^{a_n}$ for which $0\leq a_i < p$.  For any $\psi\in \Hom_L(L^{1/p},L)$, define
\[ s:=\!\!\!\! \sum_{{\mathbf{a}}\, |\, {0\leq a_i <p}}\!\! \psi(\mathbf{x}^\mathbf{a})^p \mathbf{x}^{-\mathbf{a}} .\]
Then, for any $\mathbf{b}$ with $0\leq b_i <p$, 
\[ \phi( (s \x^\mathbf{b})^{1/p}) = \phi\left( \left(\sum \psi(\mathbf{x}^\mathbf{a})^p\mathbf{x}^{\mathbf{b}-\mathbf{a}}\right)^{1/p} \right)
 = \sum\left( \psi(\mathbf{x}^\mathbf{a}) \phi \left( \mathbf{x}^{\mathbf{b}-\mathbf{a}}\right)^{1/p} \right) =\psi(\mathbf{x}^\mathbf{b}).\]
Since $\phi\circ s^{1/p}$ and $\psi$ agree on a basis for $L^{1/p}$, they coincide.
\end{proof}

\begin{Rmk}\label{gorcan}
In fact, for any Gorenstein $F$-finite ring $S$ of characteristic $p$, the module $\Hom_S(S^{1/p},S)$  is a locally free rank one  $S^{1/p}$-module, since in this case, $\Hom_S(S^{1/p},S)$ is a canonical module for $S^{1/p}$. 
The point here for Laurent rings is that a Frobenius splitting gives a canonical generator. 

One special case is for a field.
 If $\mathcal F$ is a field, then $\Hom_\mathcal F(\mathcal F^{1/p},\mathcal F)$ is a one dimensional vector space over $\mathcal F^{1/p}$, so we can take any non-zero mapping to be a basis.  In particular, if we fix a splitting $\phi:  \mathcal F^{1/p} \rightarrow \mathcal F$,
then  every  
 $\psi:  \mathcal F^{1/p} \rightarrow \mathcal F$  is  the composition 
 $\psi = \,\,\phi \circ s^{1/p}\,\,$ for some {\it unique} $s \in \mathcal F$. 
\end{Rmk}

\subsection{Frobenius splittings of upper cluster algebras}\label{StspU}
As we now prove, upper cluster algebras are {\it always} Frobenius split. Indeed, there is a natural {\bf cluster splitting} which is compatible with the cluster structure: 

\begin{Thm}\label{thm:UisSplit} Let  $\U$ be an upper cluster algebra over a  field $\k$ of positive characteristic.  For any cluster $\x=\{x_1,x_2,...,x_n\}$, the standard splitting of $\k[x_1^{\pm1},x_2^{\pm1},...,x_n^{\pm1}]$ restricts to a splitting of $\U$.  This splitting of $\U$ does not depend on the choice of cluster.
\end{Thm}

The point of the proof is the following simple but  \textbf{crucial observation}: a subalgebra $R$ of a Frobenius split  algebra $S$   is Frobenius split if $\phi(R^{1/p})\subseteq R$, where $\phi$ is some Frobenius splitting for $S$. 
\begin{proof}
Let $\x'=\{x_1',x_2,...,x_n\}$ be the mutation of $\x$ at 1, and let $P_1=x_1x_1'$ be the numerator of the mutation (see (\ref{eqmut})).  The standard splitting $\phi_\x$ of the Laurent ring $L_\x$ extends to a splitting of the fraction field $\mathcal{F}$ by localization; we check that this splitting restricts to the standard splitting $\phi_{\x'}$ on the Laurent ring $L_{\x'}$.
\[ \phi_\x((\x'^\alpha)^{1/p}) = \phi_\x((x_1'^{\alpha_1}x_2^{\alpha_2}...x_n^{\alpha_n})^{1/p})
=\phi_\x((P_1^{\alpha_1}x_1^{-\alpha_1}x_2^{\alpha_2}...x_n^{\alpha_n})^{1/p}).\]
Since $P_1$ does not contain $x_1$, the expression inside $\phi_\x$ is $x_1^{-\alpha_1}$ times a rational function of $x_2,...,x_n$.  It follows that this is zero, unless $\alpha_1=p\beta_1$ for some $\beta_1\in \mathbb{Z}$.  In this case, $\phi_\x((\x'^\alpha)^{1/p}) = \phi_\x((x_1'^{\beta_1p}x_2^{\alpha_2}...x_n^{\alpha_n})^{1/p}) 
= x_1'^{\beta_1}\phi_\x((x_2^{\alpha_2}...x_n^{\alpha_n})^{1/p})$.  Since this last expression is a Laurent monomial in $\x$, we find that
	\[ \phi_\x((\mathbf{x}'^{\alpha})^{1/p}):= \left\{ \begin{array}{cc}
		\mathbf{x}'^\beta & \alpha=p\beta \\
		0 & \text{otherwise} \\
	\end{array}	\right\},\]
and so $\phi_{\x}=\phi_{\x'}$ on $\mathcal{F}$.  Iterating this argument, we see every cluster $\x$ gives the same splitting on $\mathcal{F}$.  Since this splitting preserves each Laurent ring $L_\x$, it preserves their intersection $\U$.
\end{proof}

The cluster splitting of $\U$ inherits the universal property from Lemma \ref{lemma: uniLaurent}.

\begin{Thm}\label{Homupper}
Let $\U$ be an upper cluster algebra over a perfect field $\k$.  The cluster splitting $\phi$ of $\U$ freely generates $\Hom_\U(\U^{1/p},\U)$ as a $\U^{1/p}$-module.
\end{Thm}
\begin{proof}
Consider a $\U$-module map $\psi:\U^{1/p}\rightarrow\U$. 
 This map induces, by localization, an
$\mathcal F$-linear map $\psi:\mathcal F^{1/p}\rightarrow\mathcal F$, which we also denote (somewhat abusively) by $\psi$.  Since  $\Hom_{\mathcal F}(\mathcal F^{1/p},\mathcal F)$ is a one dimensional vector space over $\mathcal F^{1/p}$ generated by the (localization of the) standard splitting $\phi,$ we can write $\psi$ as $\phi \circ s^{1/p}$ for some {\it unique} $ s \in \mathcal F$.
We aim to show that $s \in \U$. This will complete the proof, as every $\psi \in  \Hom_\U(\U^{1/p},\U)$ will then be the composition of the standard splitting with  pre-multiplication by a unique 
$s^{1/p}$  in $\U^{1/p}$.

To show that $s \in \U$, it suffices to show that $s \in L_\x$, where $L_\x$ is the Laurent ring on \emph{any} cluster $\x=\{x_1,x_2,...,x_n\}.$ Note that  $L_\x$ is the localization of $\U$ at the cluster variables $\{x_1,x_2,...,x_n\}$. Thus the map $\psi:\U^{1/p}\rightarrow \U$ inducts an $L_\x$-module map  $\psi_\x:(L_\x)^{1/p}\rightarrow L_\x$, which we again call $\psi$. 
 By Lemma \ref{lemma: uniLaurent}, there is a unique $s_\x \in L_\x$ such that $\psi_\x(r^{1/p}) = \phi\left((s_\x r)^{1/p}\right)$ for all $r\in L_\x$.   But now, localizing further, to the fraction field $\mathcal F$, this map is of course the same as the map $\phi \circ s^{1/p}$ from the first paragraph, that is, $\phi \circ s^{1/p} = \phi \circ s_\x^{1/p}.$
  So
 by the uniqueness of $s$, we see that $s = s_{\x} \in L_\x$. Since this works for any cluster $\x$, it follows that $s \in \U$.
\end{proof}


\begin{Rmk}
The existence and universality of the cluster splitting of $\U$ are closely related to the fact that the \emph{canonical module} $\omega_{\U /\k}$ is free (cf. Remark \ref{gorcan}).  This is addressed in Appendix \ref{section: canonical}, which also describes the relation to Frobenius splittings.  
\end{Rmk}

\section{$F$-regularity of locally acyclic cluster algebras}

Strong $F$-regularity is a strengthened form of  Frobenius splitting, first introduced by   Hochster and Huneke  in \cite{HH88}. 
 Strongly $F$-regular rings have many nice properties: they are Cohen-Macaulay, normal, and  have pseudorational singularities, to name a few.
 Our main theorem in this section is that locally acyclic cluster algebras are strongly $F$-regular.
 
\subsection{Strong $F$-regularity}

Fix a domain $R$ of characteristic $p>0$. We continue to assume that $R$ is $F$-finite, meaning that 
$R^{1/p}$ is finitely generated over $R$.  This is always satisfied for  algebras finitely generated over a perfect field. 

Strong $F$-regularity will be a splitting condition on iterates of the Frobenius map. For any natural number $e$, let $F^e: R \longrightarrow R$ denote the $e$-th iterate of Frobenius, so that $F^e(r) = r^{p^e}$ for all $r \in R$. 
As in the opening paragraphs of Section \S 3, it is convenient to replace the target copy of $R$ by the canonically isomorphic ring $R^{1/p^e}$ and view the Frobenius map as the inclusion 
$$
R \hookrightarrow R^{1/p^e}
$$ inside the algebraic closure of the fraction field of $R$.

If $R \hookrightarrow R^{1/p}$ splits, it is easy to see that every iterate $R \hookrightarrow R^{1/p^e}$ splits as well. Indeed, if $ \phi: R^{1/p} \rightarrow R$ is a Frobenius splitting, then there is a naturally induced $R$-module splitting $\phi^{e}: R^{1/p^e} \rightarrow R$ induced by composition: 
$$
\xymatrix@C=1.2cm{
R^{1/p^e}\ar[r]^{(\phi)^{1/p^{e-1}}} &
R^{1/p^{e-1}} \ar[r] & \dots \ar[r] & R^{1/p} \ar[r]^{\phi} &  R. 
}
$$

In particular, upper cluster algebra also have {\bf cluster splittings} $\phi^e$  for the inclusions $\U \hookrightarrow \U^{1/p^e},$ and one easily checks (using the same proof)  that it is a generator for $\Hom_\U(\U^{1/p^e},\U)$ as a $\U^{1/p^e}$-module as in 
Theorem \ref{Homupper}. 
\medskip

\begin{defi} An $F$-finite domain $R$ is {\it strongly $F$-regular\/} if for every non-zero element $x \in R$, there exists $e \in \Bbb N$ and $\psi \in \Hom_R(R^{1/p^e}, R)$ such that $\psi(x^{1/p^e}) = 1$.
\end{defi}

Though not apparent from its definition, strong $F$-regularity is a geometric property which restricts how bad singularities can be. The next two well-known theorems are examples of this. See also \cite{SmithMSRIsurvey} for a recent survey of F-regularity.induce

\begin{Thm}\label{RegisFreg}(\cite[Theorem 3.1 c)]{HH89})
An  $F$-finite regular ring  is strongly $F$-regular.
\end{Thm}

\begin{Thm}\label{thm:FregProperties}
A Noetherian strongly $F$-regular  ring is:\begin{enumerate}
\item Frobenius split;
\item Cohen-Macaulay and normal \cite[Theorem 3.1d]{HH89};
\item pseudo-rational (cf. \cite{Smi97});
\item Kawamata log terminal  whenever it is $\Bbb Q$-Gorenstein (\cite{HW02}).
\end{enumerate}
\end{Thm}

Like most good geometric properties, strong $F$-regularity is a local condition; this is essential for our application to locally acyclic cluster algebras.

\begin{Lem}\label{lemma: localFreg}(\cite[Theorem 3.1 a)]{HH89})
A domain $R$ is strongly $F$-regular if and only if $R_\mathfrak{p}$ is strongly $F$-regular, for each prime ideal $\mathfrak{p}$.
\end{Lem}

In practice, to determine
 whether or not $R$ is strongly $F$-regular, it often suffices to check the condition in the definition for a single  element $x$. 

\begin{Prop}(\cite[Theorem 3.3]{HH89}) \label{testelement}
Let $R$ be a Noetherian $F$-finite domain which is Frobenius split.  If there is some non-zero $c\in R$ such that
\begin{enumerate}
	\item $R_c= R[c^{-1}]$ is strongly $F$-regular, and
	\item there exists $e\in \mathbb{N}$ and $\psi\in \Hom_R(R^{1/p^e},R)$ such that $\psi(c^{1/p^e})=1$,
\end{enumerate}
then $R$ is strongly $F$-regular.
%
\end{Prop}

\begin{proof} This is a well-known result lacking a precise easily accessible reference (C.f. \cite[Theorem 3.1 a]{HH89}),  so we include one here for completeness. Take any non-zero $x \in R$. By (1), there exists $\psi \in  \Hom_{R_c}({R_c}^{1/p^f},R_c)$ such that $\psi (x^{1/p^f}) = 1$. Since $\Hom_{R_c}(R_c^{1/p^n},R_c) =   \Hom_R(R^{1/p^n},R) \otimes_R R_c$, we know $\psi = \frac{1}{c^q}\widetilde{\psi} $ for some $\widetilde\psi \in \Hom_R(R^{1/p^n},R)$ and some natural number $q$, which without loss of generality, can be assumed a power of $p$. 
 So $\widetilde \psi (x^{1/p^n}) = c^q$. Now, because $R$ is Frobenius split, a splitting 
  $\phi \in  \Hom_R(R^{1/q},R)$  will send  $(c^q)^{1/q}$ to $c$. Composing this with the map given in (2) will produce a map sending $x^{1/qp^{e+n}}$ to 1. So $R$ is strongly $F$-regular.
\end{proof}

Such an element $c$ is  a {\it test element} for $R$. These types of test elements were first defined in  \cite{HH89}; for a recent survey of test elements in this context, see \cite{SmithMSRIsurvey} (more basic) or  \cite{ST11} (more advanced).

\subsection{$F$-regularity of locally acyclic cluster algebras}  We now establish the main result of this section, the $F$-regularity of locally acyclic cluster algebras.


\begin{Thm}\label{AFreg}
A locally acyclic cluster algebra $\A$ over an $F$-finite field $\k$ of prime characteristic is strongly $F$-regular.
\end{Thm}


\begin{proof} The assumption on the field ensures that $\A$ is $F$-finite.
Strong $F$-regularity is a local condition (see Lemma \ref{lemma: localFreg}), and so it can be checked on an open affine cover.  
Since locally acyclic cluster algebras admit an open affine cover by acyclic cluster algebras, it suffices to prove the theorem for acyclic cluster algebras.

Fix an \emph{acyclic} seed $(\Q,\x)$ for $\A$.  We induce on the number of mutable vertices to prove that $\A$ is strongly $F$-regular.

First, suppose  there is only one mutable variable; call it $x_1$. Then 
\[ \A = \k[ x_1,x_1',x_2^{\pm1},...,x_n^{\pm1}] / \langle x_1x_1'-p_1^+-p_1^- \rangle \]
where $p_1^+$ and $p_1^+$ are monomials in $x_2,...,x_n$ with disjoint supports.  This is a localization of the hypersurface algebra
\[ S = \k[ x_1,x_1',x_2^{},...,x_n^{}] / \langle x_1x_1'-p_1^+-p_1^- \rangle \]
Since at least one of $p_1^+$ and $p_1^-$ is not $1$,\footnote{Due to the assumption that mutable vertices must have at least one incident arrow.} the corresponding Jacobian ideal contains a monomial in $x_2,...,x_n$, and so the Jacobian ideal becomes trivial in the localization to $\A$.  Hence, $\A$ is regular, so it it strongly $F$-regular by Theorem \ref{RegisFreg}.

Assume now by induction that any acyclic quiver with $m-1$ mutable vertices defines a strongly $F$-regular cluster algebra.

Let $(\Q,\x)$ be an acyclic seed, with $m$ mutable vertices.  Since $\Q$ is acyclic, we can find a vertex which is mutable {\it and\/} admits no arrows  to any other mutable vertex---a \emph{sink}.
Label that vertex  $x_1$, and the remaining mutable vertices $x_2, \dots, x_r$.
Let $(\Q^\dagger,\x^\dagger)$ be  the `same'  seed but with $x_1$ also frozen. Since $(\Q^\dagger,\x^\dagger)$ is also acyclic, $\A(\Q^\dagger,\x^\dagger)=\U(\Q^\dagger,\x^\dagger)$ and so 
\[ \A(\Q^\dagger,\x^\dagger) = \A[x_1^{-1}] \]
is a cluster localization.  The seed $(\Q^\dagger,\x^\dagger)$ is acyclic with $m-1$ mutable vertices, and so by the inductive hypothesis, $\A[x_1^{-1}]$ is strongly $F$-regular.

Since $\A$ is acyclic, the cluster algebra $\A$ coincides with the upper cluster algebra $\U$ (Theorem \ref{AcyclicPres}), and so the cluster splitting from Theorem \ref{thm:UisSplit} is a Frobenius splitting for $\A$.  
Hence, by Proposition \ref{testelement}, it suffices to check only the test element $c:= x_1 $ in $\A$.  We construct a $\psi$ sending $x_1^{1/p^e}$ to $1$ directly, using the cluster splitting $\phi$.

Let $p_1^+$ and $p_1^-$ be the monomials appearing in the mutation formula at $x_1$, so that $x_1x_1' = p_1^++p_1^-$.  Choose $e$ large enough that each exponent appearing in $p_1^+$ or $p_1^-$ is less than $p^e$.  Since there are no arrows from $x_1$ to other mutable vertices, $p_1^-$ is a monomial only in the frozen variables; in particular, it is invertible.  Consider the map
\[ \psi = \phi^e\circ \left(\frac{x_1'}{p_1^-}\right)^{1/p^e}, \]
where $\phi^e$ is the cluster splitting of $\A \hookrightarrow \A^{1/p^e}.$
Then $$
\psi(x_1^{1/p^e}) = \phi^e\left(\left(\frac{x_1x_1'}{p_1^-}\right)^{1/p^e} \right) = 
\phi^e\left(\left(\frac{p_1^+}{p_1^-} + 1\right)^{1/p^e}\right).
$$  Since $p^e$ is greater than any exponent in the Laurent monomial $\frac{p_1^+}{p_1^-}$,  we know that $\phi^e$ kills that term, and so $\psi(x_1^{1/p^e})=1$.
By Proposition \ref{testelement}, this shows that $\A$  is strongly $F$-regular.
This completes the inductive step and the proof.
\end{proof}

\subsection{Characteristic zero consequences}

So far, our results are for cluster algebras over a field of positive characteristic.  By a standard miracle, these results imply similar consequences over fields of characteristic zero.

We first need to check that locally acyclic cluster algebras over $\ZZ$ behave as expected when tensored with a field $\k$.  Let $\A_\ZZ$ denote a cluster algebra over $\ZZ$.  Choosing any seed in $\A_\ZZ$ and replacing the cluster with a cluster over $\k$ determines a cluster algebra $\A_\k$ over $\k$, which is well-defined up to canonical isomorphism.  

\begin{Lem}
If $\A_\ZZ$ is locally acyclic, then $\A_\k$ is locally acyclic and $\k\otimes_\ZZ\A_\ZZ\simeq \A_\k$.
\end{Lem}
\begin{proof}
If $\A_\ZZ$ is acyclic, then any acyclic seed of $\A_\ZZ$ corresponds to an acyclic seed of $\A_\k$ with the same quiver.  The presentations of $\A_\ZZ$ and $\A_\k$ from Proposition \ref{AcyclicPres} coincide except for the ring of scalars, and so $\k\otimes_\ZZ\A_\ZZ\simeq \A_\k$.

If $\A_\ZZ$ is locally acyclic, let $\{(\A_i)_\ZZ\}_{i\in I}$ be a cover by acyclic cluster algebras.  By the previous paragraph, $\k\otimes_\ZZ(\A_i)_\ZZ \simeq (\A_i)_\k$ is an acyclic cluster localization of $\A_\k$.  Since extension of scalars sends covers to covers, $\{(\A_i)_\k\}_i$ is a cover of $\A_\k$.  Since the map $\k\otimes_\ZZ\A_\ZZ\rightarrow \A_\k$ is locally an isomorphism, it is an isomorphism.
\end{proof}

With this in hand, we may prove one of our main theorems.

\begin{Thm}
A locally acyclic cluster algebra over a field $\k$ of characteristic zero has (at worst) canonical singularities.
\end{Thm}
\begin{proof}
Let $\A$ be a locally acyclic cluster algebra over $\k$, and let $\A_\ZZ$ be the corresponding locally acyclic cluster algebra over $\ZZ$.  By the preceding Lemma, for any prime $p\in \ZZ$, 
\[ \mathbb{F}_p\otimes_\ZZ \A_\ZZ \simeq \A_{\mathbb{F}_p}\]
is locally acyclic.
By Theorem \ref{AFreg}, $\A_{\mathbb{F}_p}$ is strongly $F$-regular.  On the other hand, $\A$ is Gorenstein by Corollary \ref{propLA}.  Thus, by the main theorem of  \cite[Theorem 4.3]{Smi97}, 
$\A_\k$ has (at worst) \emph{rational singularities}. But Gorenstein rational singularities are  canonical (cf. the discussion in \cite{Elk81}).
\end{proof}

%
%
%
%

\section{The non-locally-acyclic setting}

What can be said about cluster algebras and upper cluster algebras which are not locally acyclic?  We provide examples which demonstrate that strong $F$-regularity is still possible, but not necessary. We also support the general philosophy that $\U$ should be better-behaved than $\A$ by proving that $F$-regularity of $\A$ implies $F$-regularity of $\U$. 

We end this section by showing that a related algebra, the \emph{lower bound algebra}, is always Frobenius split. We do not know whether or not the lower bound algebra is always strongly $F$-regular.

\subsection{$F$-regularity of $\A$ implies $F$-regularity $\U$}

In this section, we consider a completely arbitrary cluster algebra $\A$  (possibly infinitely generated) over a perfect field.


\begin{Lem}\label{lem:extend} 
Fix any integer $e\geq 1$ and let $\varphi \in \Hom_{\mathcal{A}}(\A^{1/p^e},\A)$. Then $\varphi$ extends uniquely to a map in $\Hom_{\U}(\U^{1/p^e},\U)$.
\end{Lem}

\begin{proof}
The map $\varphi$ extends, by localization, to the Laurent ring generated by the cluster variables in any given cluster. Hence, it preserves $\U$, the common intersection of all of these Laurent rings.
\end{proof}




\begin{Prop}\label{AFregUFreg}
If $\A$ is strongly $F$-regular then $\U$ is strongly $F$-regular.
\end{Prop}

\begin{proof}
Suppose that $c$ is a non-zero element of $\U$. Since $\A$ and $\U$ have the same fraction field, there is an $a\in \A$ for which $ac$ is a non-zero element of $\A$. Because $\A$ is strongly $F$-regular, there is an integer $e\geq 1$ and a map $\varphi \in \Hom_{\A}(\A^{1/p^e},\A)$ for which $\varphi((ac)^{1/p^e}) = 1$. By Lemma \ref{lem:extend}, we may extend $\varphi$ uniquely to a map $\widetilde{\varphi}:\U^{1/p^e}\rightarrow \U$. 

Let $m_a:\U^{1/p^e}\rightarrow \U^{1/p^e}$ be the multiplication map given by $m_a(x^{1/p^e}) = a^{1/p^e}x^{1/p^e}$. Then the composition $\widetilde{\varphi}\circ m_a$ is an element of $\Hom_{\U}(\U^{1/p^e},\U)$ which maps $c^{1/p^e}$ to $1$.
\end{proof}

In what follows, we focus on the $F$-regularity of upper cluster algebras.

\subsection{The Markov upper cluster algebra}

Consider the seed $(\Q,\x)$ defined in Figure \ref{fig: Markov}.  Observe that it has three mutable vertices and no frozen vertices.

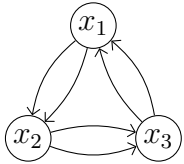
\begin{figure}[h!t]	
	\begin{tikzpicture}
	\begin{scope}[xshift=0in]
		\path[use as bounding box] (-1,-1) rectangle  (1,1.25);
		\node[mutable] (x) at (90:1) {$x_1$};
		\node[mutable] (y) at (210:1) {$x_2$};
		\node[mutable] (z) at (-30:1) {$x_3$};
		\draw[-angle 90,relative, out=15,in=165] (x) to node[inner sep=.1cm] (xy1m) {} (y);
		\draw[-angle 90,relative, out=-20,in=-160] (x) to node[inner sep=.1cm] (xy2m) {} (y);
		\draw[-angle 90,relative, out=15,in=165] (y) to node[inner sep=.1cm] (yz1m) {} (z);
		\draw[-angle 90,relative, out=-20,in=-160] (y) to node[inner sep=.1cm] (yz2m) {} (z);
		\draw[-angle 90,relative, out=15,in=165] (z) to node[inner sep=.1cm] (zx1m) {} (x);
		\draw[-angle 90,relative, out=-20,in=-160] (z) to node[inner sep=.1cm] (zx2m) {} (x);
	\end{scope}
	\end{tikzpicture}
\caption{The seed for the Markov cluster algebra in $\mathcal{F}=\k(x_1,x_2,x_3)$.}
\label{fig: Markov}
\end{figure}

Introduced in \cite{BFZ05}, the \emph{Markov cluster algebra} $\A= \A(\Q,\x)$ is a standard source of counterexamples and pathologies.  
For example, $\A=\A(\x,\B)$ is not a locally acyclic cluster algebra, and indeed $\A\subsetneq \U$ \cite[Theorem 1.26]{BFZ05}. Moreover, the Markov cluster algebra $\A$ is not Noetherian \cite{MulLA}. 

Nevertheless, the Markov upper cluster algebra $\U=\U(\Q,\x)$ is quite well-behaved.  It was shown in \cite{MM13} that it can be presented as the hypersurface algebra
\[ \U\cong \k[x_1,x_2,x_3,M]/\langle x_1x_2x_3M-x_1^2-x_2^2-x_3^2\rangle \]
Equivalently, the upper cluster algebra $\U$ is generated inside the field $\mathcal{F}$ by $x_1,x_2,x_3$ and the element $M=\frac{x_1^2+x_2^2+x_3^2}{x_1x_2x_3}$.

\begin{Prop}
If $char(\k)\neq2,3$, then the Markov upper cluster algebra $\U$ is strongly $F$-regular.
\end{Prop}
\begin{proof}
Since $\U$ is Frobenius split by Theorem \ref{thm:UisSplit}, and the localization of $\U$ at $x_1x_2x_3$ is a Laurent ring, $x_1x_2x_3$ is a test element for $\U$. 
Consider now the morphism $\varphi:\U^{1/p^e}\rightarrow\U^{1/p^e}$ defined by $\varphi(-)=\phi^{e}((1/6\cdot M^3)^{1/p^e}-)$, where $\phi^{e}$ is the iterated cluster splitting of $\U$ defined in \ref{StspU} and $M$ is described as before. This morphism $\varphi$ is mapping $(x_1x_2x_3)^{1/p^e}$ to 1:
$$\varphi((x_1x_2x_3)^{1/p^e})=\phi\Big(\Big(1/6\cdot\Big(\frac{x_1^2+x_2^2+x_3^2}{x_1x_2x_3}\Big)^3x_1x_2x_3\Big)^{1/p^e}\Big)=$$
$$=\phi\Big(\Big(1/6\cdot \frac{\sum_{i+j\leq 3,i,j\neq 1}c_{i,j}x_1^{2i}x_2^{2j}x_3^{6-2i-2j}+6x_1^{2}x_2^{2}x_3^{2}}{(x_1x_2x_3)^2}\Big)^{1/p^e}\Big)=1,$$
where the $c_{i,j}$ are some combinatorial coefficients.
This shows that $\U$ is strongly $F$-regular.
\end{proof}

\subsection{A non-$F$-regular upper cluster algebra}

Generalizing the previous setting, consider the seed $(\Q,\x)$ defined in Figure \ref{fig: GMarkov} for some integer $a\geq2$.  

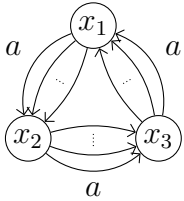
\begin{figure}[h!t]	
	\begin{tikzpicture}
	\begin{scope}[xshift=0in]
		\path[use as bounding box] (-1,-1) rectangle  (1,1.25);
		\node[mutable] (x) at (90:1) {$x_1$};
		\node[mutable] (y) at (210:1) {$x_2$};
		\node[mutable] (z) at (-30:1) {$x_3$};
		\draw[-angle 90,relative, out=15,in=165] (x) to node[inner sep=.1cm] (xy1m) {} (y);
		\draw[-angle 90,relative, out=-20,in=-160] (x) to node[inner sep=.1cm] (xy2m) {} (y);
		\draw[-angle 90,relative, out=-40,in=-140] (x) to node[above left] {$a$} (y); 
		\draw[densely dotted] (xy1m) to (xy2m);
		\draw[-angle 90,relative, out=15,in=165] (y) to node[inner sep=.1cm] (yz1m) {} (z);
		\draw[-angle 90,relative, out=-20,in=-160] (y) to node[inner sep=.1cm] (yz2m) {} (z);
		\draw[-angle 90,relative, out=-40,in=-140] (y) to node[below] {$a$}  (z);
		\draw[densely dotted] (yz1m) to  (yz2m);
		\draw[-angle 90,relative, out=15,in=165] (z) to node[inner sep=.1cm] (zx1m) {} (x);
		\draw[-angle 90,relative, out=-20,in=-160] (z) to node[inner sep=.1cm] (zx2m) {} (x);
		\draw[-angle 90,relative, out=-40,in=-140] (z) to node[above right] {$a$}  (x);
		\draw[densely dotted] (zx1m) to  (zx2m);
	\end{scope}
	\end{tikzpicture}
\caption{A seed in $\mathcal{F}=\k(x_1,x_2,x_3)$, where $a\geq 2$.}
\label{fig: GMarkov}
\end{figure} 

Let $\U=\U(\Q,\x)$ denote the associated upper cluster algebra. As shown in \cite{MM13}, this \emph{generalized Markov upper cluster algebra} can be presented as
\[ \U \cong \k[x_1,x_2,x_3,M]/\langle x_1x_2x_3M - x_1^a-x_2^a-x_3^a\rangle. \]

\begin{Prop}
If $a\geq 3$, then $\U$ is not strongly $F$-regular.
\end{Prop}
\begin{proof}
Notice that $\U$ is graded, with
\[ \deg(x_1)=\deg(x_2)=\deg(x_3) =1,\;\;\; \deg(M) = a-3 \]
When $a\geq 3$, every homogeneous element in $\U$ has degree at least $0$.  As a consequence, the span of the positive degree elements forms a non-zero ideal $I$.

The cluster splitting $\phi$ sends positive degree elements to positive degree elements or zero, so $\phi^e(I^{1/p^e}) \subseteq I$ for any $e$.  By Theorem \ref{Homupper}, any $\psi\in \Hom_\U(\U^{1/p^e},\U)$ can be written as
\[ \psi = \phi^e(s^{1/p^e} - ) \]
for some $s\in \U$.  Since $(sI)^{1/p^e} \subseteq I^{1/p^e}$, we see that $\psi(I^{1/p^e})\subseteq I^{1/p^e}$ for any $\psi\in \Hom_\U(\U^{1/p^e},\U)$.  In particular, for any $c\in I$, there is no $\psi \in \Hom_\U(\U^{1/p^e},\U)$ such that $\psi(c)=1\not\in I$, and so $\U$ is not strongly $F$-regular.
\end{proof}

%
%

By Proposition \ref{AFregUFreg}, this extends to the cluster algebra as well.

\begin{Cor}
If $a\geq 3$, then $\A(\Q,\x)$ is not strongly $F$-regular.
\end{Cor}

\begin{Rmk}
The positive degree elements in $\U$ form an ideal $(x_1, x_2, x_3)$ stable under all maps in $ \Hom_\U(\U^{1/p^e},\U)$, that is, compatible with respect to every element of $ \Hom_\U(\U^{1/p^e},\U)$. This is the {\it test ideal} of $\U$.  See \cite{ST11}.
\end{Rmk}

\subsection{Lower bound algebras}

Fix a seed $(\Q,\x)$, where $\x = (x_1,x_2,\dots,x_n)$. As before, let 
\[p_i^{+} := \prod_{j\rightarrow i}x_j,  \textrm{ and  }~~~ p_i^{-} := \prod_{j \leftarrow i} x_j,\] 
and let $x_i' := (p_i^++p_i^-)x_i^{-1}$. The algebra $\mathcal{L}(\Q,\x)$ defined by
\[
\mathcal{L}(\Q,\x):=\k[x_1,x_2,\dots,x_n,x_1',x_2',\dots,x_n']\subseteq \mathcal{F}=\k(x_1,\dots,x_n)
\]
is called the \textbf{lower bound algebra} associated to the seed $(\Q,\x)$. Notice that $\mathcal{L}(\Q,\x)\subseteq \mathcal{A}(\Q,\x)$. This inclusion is an equality if and only if $\Q$ is an acyclic quiver (cf. \cite[Theorem 1.20]{BFZ05}). 

\begin{Lem}
The kernel $L$ of the surjective ring homomorphism 
\[ \k[x_1,\dots,x_n,y_1,\dots,y_n]\longrightarrow \mathcal{L}(\Q,\x),~~~ x_i \mapsto x_i, ~~~ y_i \mapsto x_i' \]
is a prime component of the ideal $I := \langle x_1y_1- (p_1^{+} + p_1^{-}), \dots, x_ny_n-(p_n^+ + p_n^-)\rangle$. 
\end{Lem}

\begin{proof}
Since $\mathcal{L}(\Q,\x)$ is a domain, $L$ is a prime ideal. To see that $L$ is a component of $I$, let $S = \k[x_1,\dots,x_n,y_1,\dots,y_n]$, and observe that 
\[
I S[(x_1\cdots x_n)^{-1}] = \langle y_1-x_1',\dots,y_n-x_n'\rangle.
\]
Since $L = \langle y_1-x_1',\dots,y_n-x_n'\rangle\cap S$, it follows that $(I S[(x_1\cdots x_n)^{-1}])\cap S = L$, and thus $L$ is a prime component of $I$. 
\end{proof}




\begin{Prop}
The lower bound algebra $\mathcal{L}(\Q,\x)$ is Frobenius split. 
\end{Prop}

\begin{proof}
Fix any prime $p>0$, let $S := \k[x_1,\dots,x_n,y_1,\dots,y_n]$, and let $\mathcal{B}$ denote the $S$-module basis of $S^{1/p}$ consisting of all those monomials 
\[x_1^{a_1/p}\cdots x_n^{a_n/p}y_1^{a_{n+1}/p}\cdots y_n^{a_{2n}/p}, ~~~0\leq a_i < p.\] 

Define $\psi: S^{1/p}\longrightarrow S$ to be the $S$-linear map which takes value $1$ on the basis element $x_1^{(p-1)/p}\cdots x_n^{(p-1)/p}y_1^{(p-1)/p}\cdots y_n^{(p-1)/p}$ and $0$ on all other elements of $\mathcal{B}$. 

We will construct a Frobenius splitting of $S$ which descends to a Frobenius splitting of $\mathcal{L}$. To this end, let
\[f = \prod_{1\leq i\leq n}(x_iy_i-p_i^{+}-p_i^{-}),\]
and observe that pre-multiplication of $\psi$ by $f^{(p-1)/p}$ is a Frobenius splitting of $S$. Indeed, notice that the only monomial of $f^{(p-1)/p}$ with all of the variables $x_1,\dots, y_n$ is $x_1^{(p-1)/p}\cdots x_n^{(p-1)/p}y_1^{(p-1)/p}\cdots y_n^{(p-1)/p}$, which gets sent to $1$ by $\psi$. So, $\psi(f^{(p-1)/p}\cdot 1)=1$, and $\psi(f^{(p-1)/p}-)$ is thus a Frobenius splitting. Furthermore, if $J = \langle f \rangle$, then 
\[\psi(f^{(p-1)/p}J^{1/p})\subseteq J. \]
That is, $J$ is a compatibly split ideal. Because sums and prime components of compatibly split ideals are compatibly split (cf., for eg., \cite[Chapter 1.2]{BK05}), the ideal $L\subseteq S$ that defines the lower bound algebra $\mathcal{L}(\Q,\x)$ is compatibly split. The Frobenius splitting $\psi(f^{(p-1)/p}-): S^{1/p}\longrightarrow S$ therefore descends to a Frobenius splitting of the lower bound algebra.
\end{proof}

\appendix

\section{The canonical module of an upper cluster algebra}\label{section: canonical}

This appendix considers the canonical module of an upper cluster algebra $\U$ over a field $\k$.\footnote{The results remain true over $\ZZ$.} Since upper cluster algebras need not be Noetherian \cite{Speyer}, we must be careful which definition we use.  

\subsection{Canonical modules}

Let $S$ be a normal domain over $\k$, such that the fraction field $\mathcal{F}(S)$ has transcendence degree $n$ over $\k$.
Define the \textbf{canonical module} of $S$ over $\k$ to be the $S$-module
\[ \omega_{S/\k}:= (\Lambda^n_S \Omega_{S/\k})^{**}.\]
If $S$ is regular  (such as a field), then the double dual in the definition is unnecessary, and $\omega_{S/\k} = \Lambda^n_S\Omega_{S/\k}$.
This construction commutes with localization; in particular, there is a natural embedding 
\[\omega_{S/\k}\subseteq \omega_{\mathcal{F}(S)/\k}= \Lambda_{\mathcal{F}(S)}^n \Omega_{\mathcal{F}(S)/\k}.\]
 into the canonical module of the fraction field.  
%
%

\subsection{The log volume form}
Let $\U$ be an upper cluster algebra over $\k$.  The algebra $\U$ is normal and the transcendence degree of its fraction field over $\k$ is the rank $n$.  For a cluster $\x$ with functions $\{x_1,x_2,...,x_n\}$ indexed by $\{1,2,...,n\}$, let $L_\x$ denote the $\k$-Laurent ring in that cluster, and define the \textbf{log volume form}
\[ \mu_\x := \frac{dx_1\wedge dx_2\wedge \cdots \wedge dx_n}{x_1x_2\cdots x_n} \in \omega_{L_\x/\k}.\]
Note that a permutation of the indices may change the sign of this element. 

\begin{Prop}\label{prop: cangenLaurent}
The canonical module $\omega_{L_\x/\k}$ is free of rank one over $L_\x$, and generated by the log volume form $\mu_\x$.
\end{Prop}
The log volume form is an invariant of the cluster algebra, up to sign.
\begin{Prop}
For two different clusters $\x,\y$ of $\U$,  we have $\mu_\x=\pm \mu_\y$.
\end{Prop}
\begin{proof}
It suffices to check the proposition for a single mutation.  Let $\x'=\{x_1,x_2,...,x_i',...,x_n\}$, where $x_i' = \frac{p_i^++p_i^-}{x_i}$.
\[ \frac{dx_i'}{x_i'} = \frac{d(p_i^++p_i^-)}{x_ix_i'} - \frac{(p_i^++p_i^-)dx_i}{x_i^2x_i'} =
\frac{d(p_i^++p_i^-)}{p_i^++p_i^-} - \frac{dx_i}{x_i}\]
Since $p_i^+$ and $p_i^-$ are monomials in $\{x_1,x_2,...,x_{i-1},x_{i+1},...,x_n\}$,
\begin{eqnarray*}
\frac{dx_1\wedge \cdots\wedge dx_i'\wedge\cdots \wedge dx_n}{x_1\cdots x_i'\cdots x_n} &=& 
- \frac{dx_1\wedge \cdots\wedge dx_i\wedge\cdots \wedge dx_n}{x_1\cdots x_i\cdots x_n}
\end{eqnarray*}
Hence, $\mu_\x=-\mu_{\x'}$.  Iterating mutations or permuting the indices will change this form by at most a sign.
\end{proof}

\subsection{Canonical modules of upper cluster algebras}

Since either log volume form freely generates the canonical module after localizing to a cluster Laurent ring, it follows that they \emph{freely generate} the canonical module of $\U$.

\begin{Thm}\label{thm: loggen}
For an upper cluster algebra $\U$ over a field\footnote{The theorem remains true when $\k$ is an arbitrary normal domain.} $\k$, the canonical module 
$\omega_{\U/\k}$
is free of rank one over $\U$, and generated by a log volume form in any cluster.
\end{Thm}
\begin{proof}
Fix a log volume form $\mu$ in some cluster.
For any cluster $\x$, the localization $L_\x\otimes \Lambda_\U^n\Omega_{\U/\k}= L_\x \mu$ by Proposition \ref{prop: cangenLaurent}.  Let 
\[ \Lambda_\U^n\Omega_{\U/\k}\rightarrow f(\Lambda_\U^n\Omega_{\U/\k})\] 
be the quotient by the maximal torsion submodule.  Then $f(\Lambda_\U^n\Omega_{\U/\k})$ is contained inside the localization $L_\x\mu$, which is contained inside the generic canonical module $\Lambda^n_{\mathcal{F}(\U)}\Omega_{\mathcal{F}(\U)/\k}$.  Intersecting over all clusters, we obtain a map
\[\Lambda_\U^n\Omega_{\U/\k} \rightarrow f(\Lambda_\U^n\Omega_{\U/\k}) \subseteq \bigcap_{\text{clusters } \x}\left( L_\x\mu\right) = \left(\bigcap_{\text{clusters } \x } L_\x\right) \mu = \U\mu.\]
Define $\mu^*\in (\Lambda_\U^n\Omega_{\U/\k})^*=\Hom_\U(\Lambda_\U^n\Omega_{\U/\k},\U)$ to be the composition of the above map with the $\U$-module map $\U \mu\rightarrow \U$ which sends $\mu$ to $1$.

Consider another $\U$-module map $\psi:\Lambda_\U^n\Omega_{\U/\k}\rightarrow \U$.  Since $\U$ is torsion-free, $\psi$ factors through $f(\Lambda_\U^n\Omega_{\U/\k})$.  Localizing $\psi$ at a cluster $\x$ gives an $L_\x$-module map $\psi_\x:L_\x \mu\rightarrow L_\x$.  Let $s_\x:= \psi_\x(\mu)$, and note that $\psi_\x(a\mu) = as_\x$ for all $a\in L_\x$.

Localizing at a different cluster $\y$ gives a map $\psi_\y:L_\y\mu\rightarrow L_\y$, which is of the form $\psi_\y(a\mu) = as_\y$ for some $s_\y\in L_\y$.  Since there is some non-zero $b\in U$ such that $b\mu \in f(\Lambda_\U^n\Omega_{\U/\k})$ (the product of the variables in any cluster suffices),
\[ bs_\x = b\psi_\x(\mu) = \psi(b\mu) = b\psi_\y(\mu) = bs_\y,\]
it follows that $s_\x=s_\y$ in $L_\x\cap L_\y$.  Repeating for all clusters, $s_\x\in \U$.  Therefore, $\psi(a\mu) = as_\x= \mu^*(s_\x a\mu)$ for all $a\in \U$; this proves that $\mu^*$ freely generates the dual module
\[ (\Lambda_\U^n\Omega_{\U/\k})^* = \U \mu^* .\]
Dualizing both sides demonstrates that $\omega_{\U/\k}= \U \mu$.
\end{proof}
There are examples where the log volume forms are not in $\Omega^n_{\U/\k}$; hence, the double dual in the definition of $\omega_{\U/\k}$ is necessary.

\begin{Cor}
A Noetherian upper cluster algebra over a field is Gorenstein.
\end{Cor}

\subsection{Canonical modules and Frobenius splittings}

We sketch the relation between canonical modules and Frobenius splittings here; further details may be found in \cite[Section 1.3]{BK05}.

Let $\k$ be a field of positive characteristic $p\neq2$, and let $X$ be a smooth, locally finite-type scheme over $\k$.  The Frobenius map becomes a flat, finite morphism
\[ f:X\rightarrow X. \]
The push-forward functor $f_*:Coh(X)\rightarrow Coh(X)$ then has a right-adjoint $f^!:Coh(X)\rightarrow Coh(X)$, together with an adjunction map
\[ tr:f_*f^!\rightarrow Id\]
called the \emph{trace}.

The coherent sheaf $f^!(\mathcal{O}_X)$ and its trace map can be connected with Frobenius splittings as follows.  On any open affine subscheme $\Spec(R)\subseteq X$, 
\begin{itemize}
	\item $f^!(\mathcal{O}_X)[\Spec(R)]$ is isomorphic to $\Hom_R(R^{1/p},R)$ as an $R^{1/p}$-module; equivalently, 
	to $\Hom_{R^p}(R,R^p)$ as an $R$-module, 
	\item $f_*f^!(\mathcal{O}_X)[\Spec(R)]$ is isomorphic to $\Hom_R(R^{1/p},R)$ as an $R$-module, and
	\item the trace map is given by the $R$-module map
	\[ \Hom_R(R^{1/p},R)\rightarrow R\]
	which sends a map $f:R^{1/p}\rightarrow R$ to $f(1)\in R$.
\end{itemize}
Hence, the sheaf of Frobenius splittings is isomorphic to $tr^{-1}(1)\subset f_*f^!(\mathcal{O}_X)$.  

Duality theory for the morphism $f$ gives an alternate description of $f^!(\mathcal{O}_X)$, in terms of the canonical sheaf $\omega_{X/\k}$.
\begin{Thm}\cite[Section 1.3]{BK05}\label{thm: duality}
There are natural isomorphisms\footnote{The negative exponent on $(\omega_{X/\k})^{1-p}$ denotes a positive power of the dual sheaf $\omega_{X/\k}^*$.  This is a markedly different use of exponential notation than $R^{1/p}$.}
\[ f^!(\mathcal{O}_X) \stackrel{\sim}{\longrightarrow} \mathcal{H}om_X(f^*\omega_{X/\k},\omega_{X/\k})\stackrel{\sim}{\longrightarrow} (\omega_{X/\k})^{1-p}, \]
\[ f_*f^!(\mathcal{O}_X) \stackrel{\sim}{\longrightarrow} \mathcal{H}om_X(\omega_{X/\k},f_*\omega_{X/\k})\stackrel{\sim}{\longrightarrow} f_*((\omega_{X/\k})^{1-p}), \]
and a map $\tau:f_*\omega_{X/\k}\rightarrow \omega_{X/\k}$, such that the trace map is given by
\[ f_*f^!(\mathcal{O}_X) \stackrel{\sim}{\longrightarrow} \mathcal{H}om_X(\omega_{X/\k},f_*\omega_{X/\k})\stackrel{\tau\circ -}{\longrightarrow} \mathcal{H}om_X(\omega_{X/\k},\omega_{X/\k}) \stackrel{\sim}{\longrightarrow} \mathcal{O}_X.\]
\end{Thm}

\subsection{Back to cluster algebras}

Let $\U$ be an upper cluster algebra.  For any seed $(\Q,\x)$, the Laurent phenomenon says that freezing every mutable vertex gives the localization $\k[x_1^{\pm1},...,x_n^{\pm1}]$ in the cluster $\x=\{x_1,...,x_n\}$.  Geometrically, this means there is an open subscheme
\[ \text{Spec}(\k[x_1^{\pm1},...,x_n^{\pm1}]) \subseteq \text{Spec}(\U).\]
Let us call subschemes of this form \textbf{cluster tori}.
Let $X\subseteq \text{Spec}(\U)$ be the \emph{union of the cluster tori}, as $(\Q,\x)$ runs over all seeds.

While the scheme $\text{Spec}(\U)$ is generally neither smooth nor locally finite type over $\k$, the open subscheme $X$ is both.  Hence, by Theorem \ref{thm: duality}, we have isomorphisms
\[ f^!(\mathcal{O}_X) \stackrel{\sim}{\longrightarrow} \mathcal{H}om_X(f^*\omega_{X/\k},\omega_{X/\k})\stackrel{\sim}{\longrightarrow} (\omega_{X/\k})^{1-p} .\]

\begin{Prop}
Let $X$ be the union of the cluster tori in $\text{Spec}(\U)$.
\begin{enumerate}
	\item $\mathcal{O}_X(X)$ is isomorphic to $\U$ as a $\k$-algebra.
	\item $f^!(\mathcal{O}_X)(X)$ is isomorphic to $\Hom_{\U^p}(\U,\U^p)$ as a $\U$-module.
	\item $\omega_{X/\k}(X)$ is isomorphic to $\omega_{\U/\k}$ as a $\U$-module.
\end{enumerate}
On any cluster torus, these isomorphisms restrict to the obvious isomorphisms.
\end{Prop}
\begin{proof}
The first isomorphism is a standard fact about upper cluster algebras; see e.g. \cite[Prop. 3.4]{MM13}.  The other two isomorphisms follow from Theorems \ref{Homupper} and \ref{thm: loggen}, which show that $Hom_{\U^p}(\U,\U^p)$ and $\omega_{\U/\k}$ are each free of rank 1 over $\U$ with a distinguished generator (the cluster splitting and either log volume form, respectively).  On each cluster torus, the sheaves $f^!(\mathcal{O}_X)$ and $\omega_{X/\k}$ are free and generated by the restriction of the generator.  Hence, a global section of $f^!(\mathcal{O}_X)$ or $\omega_{X/\k}$ can be written as the distinguished generator, times a rational function which is Laurent in each cluster; that is, an element of $\U$.  
\end{proof}

As a consequence, we have an isomorphism of $\U$-modules
\[ \Hom_{\U^p}(\U,\U^p) \stackrel{\sim}{\longrightarrow} (\omega_{\U/\k})^{1-p},\]
where $\omega_{\U/\k}$ to a negative power means $\omega_{\U/\k}^*=(\Lambda^n_\U\Omega_{\U/\k})^*$ to a positive power.

The connection between cluster splittings and log volume forms starts to become clear.  Theorem \ref{thm: loggen} establishes that $\omega_{\U/\k}$ is free of rank 1 as a $\U$-module.  Hence, $\Hom_{\U}(\U,\U^p)$ is free of rank 1 as a $\U$-module, or equivalently, $\Hom_{\U}(\U^{1/p},\U)$ is free of rank 1 as a $\U^{1/p}$-module.  

To choose a distinguished generator, we observe that $\omega_{\U/\k}$ has two natural generators (the log volume forms) which differ by a sign.  Since $p$ is odd, then the $(1-p)$-th power of the two log volume forms coincide, so $(\omega_{\U/\k})^{1-p}$ has a canonical generator.  This determines a canonical generator in $\Hom_{\U}(\U^{1/p},\U)$ over $\U^{1/p}$; all that remains is to observe that it coincides with the cluster splitting.
\begin{Prop}
If $\mu\in \omega_{\U/\k}$ is either log volume form, then the image of $\mu^{1-p}$ under the map
\[ (\omega_{\U/\k})^{1-p}\stackrel{\sim}{\longrightarrow} \Hom_{\U}(\U^{1/p},\U)\]
is the cluster spitting $\phi:\U^{1/p}\rightarrow \U$.  
\end{Prop} 
\noindent The reader is cautioned that, as written, this is not a module map; rather, it intertwines the $\U$-action on $(\omega_{\U/\k})^{1-p}$ and the $\U^{1/p}$-action on $\Hom_\U(\U^{1/p},\U)$.
\begin{proof}
If $L_\x$ is the $\k$-Laurent ring in some cluster $\x$, the localization map $\U\subset L_\x$ induces localization maps
\[ (\omega_{\U/\k}) ^{1-p}\hookrightarrow (\omega_{L_\x/\k})^{1-p},\;\;\; and \ \Hom_\U(\U^{1/p},\U) \hookrightarrow \Hom_{L_\x}(L_\x^{1/p},L_\x).\]
It suffices to check that $\mu^{1-p}$ is sent to the standard splitting of $L_\x$; this is essentially \cite[Lemma 1.3.6]{BK05}.
\end{proof}


%
%
%
%

\bibliography{UniversalBib}
\bibliographystyle{amsalpha}

\end{document}